\documentclass[11pt]{article}

\usepackage{amsmath}
\usepackage{latexsym}
\usepackage{amsfonts}
\usepackage{amssymb}
\usepackage{amscd}
\usepackage{theorem}
\usepackage{color}

\newtheorem{theo}{Theorem}[section]
\newtheorem{definition}[theo]{Definition}
\newtheorem{lemma}[theo]{Lemma}
\newtheorem{proposition}[theo]{Proposition}
\newtheorem{corollary}[theo]{Corollary}

{\theorembodyfont{\rmfamily}
  \newtheorem{example}[theo]{Example}

  \newtheorem{remark}[theo]{Remark}
  
}

\newenvironment{proof}{    
  \noindent
  \textbf{Proof.}}{
  \hfill $\Box$
  \vspace{3mm}
}

\numberwithin{equation}{section}

\newcommand{\N}{\mathbb{N}} 
\newcommand{\R}{\mathbb{R}} 
\newcommand{\C}{\mathbb{C}} 

\DeclareMathOperator{\supp}{supp}

\newcommand{\eps}{\varepsilon}

\begin{document}

\title{Composition operators on the Schwartz space}
\author{Antonio Galbis, Enrique Jord\'a}
\maketitle
\begin{abstract}
 We study composition operators on the Schwartz space of rapidly decreasing functions. We prove that such a composition operator is never a compact operator and we obtain necessary or sufficient conditions for the range of the composition operator to be closed. These conditions are expressed in terms of multipliers for the Schwartz class and the closed range property of the corresponding operator considered in the space of smooth functions.
 
\end{abstract}

\section{Introduction.}
Composition operators, as well as multiplication operators, are  widely studied in the last years mainly in spaces of holomorphic functions. One can check \cite{cowen,shapiro} and the references given therein. Much less seems to be known about the functional analytic properties of the composition operator when it is defined in real spaces of smooth functions. To be more precise let us denote by $C^{\infty}(\R^d)$ the Fr\'echet space of all smooth functions endowed with its natural topology of uniform convergence of the derivatives on the compact sets of $\R^d.$ In the case that $\varphi:\R^d\to\R^d$ is an analytic function, properties of the composition operator $C_{\varphi}:C^{\infty}(\R^d)\to C^{\infty}(\R^d),\ f\mapsto f\circ \varphi,$ are well known since long (\cite{glaeser,malgrange,tougeron}). However, when the symbol $\varphi$ is only assumed to be smooth, the operator is not well understood. In particular, to characterize when the operator has closed range becomes a tough problem. Yet, there are 
some recent important advances in this direction concerning the one variable case. Kenessey and Wengenroth characterized in \cite{kw} the injective symbols $\varphi:\R\to\R^d$ such that $C_\varphi:C^\infty(\R^d)\to C^\infty(\R)$ has closed range. A sufficient condition on the symbol $\varphi:\R\to\R$ for the operator $C_{\varphi}:C^\infty(\R)\to C^{\infty}(\R)$ to have closed range was obtained by Przestacki. This condition is also necessary under some mild assumptions on $\varphi$ (\cite{przstacki1,przstacki2,przstacki3}).  A characterization of those classes of ultradifferentiable functions which are closed under composition was obtained in \cite{fg}.

Besides the space of smooth functions and the space of real analytic functions, one of the most important spaces in mathematical analysis is the Schwartz space $S(\R^d)$ of rapidly decreasing functions. The multipliers in $S(\R^d)$ are the functions $F:\R^d\to \R^d$ such that the multiplication operator $M_F:S(\R^d)\to S(\R^d),\ f\mapsto Ff,$ is well defined and continuous. The set of all multipliers is denoted by $O_M(\R^d)$ and consists of those smooth functions whose 
derivatives of arbitrary order have polynomial growth at infinity. Bonet, Frerick and the 
second author have characterized in \cite{bfj2012} the multipliers $\varphi\in O_M$($:=O_M(\R)$) such that $M_{\varphi}: S(\R)\to S(\R),\ f\mapsto f\varphi,$ has closed range. 

 In this paper we study composition operators defined in the Schwartz space $S(\R)$ of one variable rapidly decreasing functions. We characterize the smooth functions $\varphi:\R\to \R$ for which the composition operator $C_\varphi:S(\R)\to S(\R), f\mapsto f\circ \varphi,$ is well defined and continuous and, if so, we study some of its properties. In particular, we prove that $C_\varphi$ can never be a compact operator and obtain necessary or sufficient conditions for the range of the composition operator to be closed in $S(\R)$. These conditions are expressed in terms of the multiplication operator studied in \cite{bfj2012} and the closed range of the corresponding operator considered in the space of smooth functions, involving then the conditions considered by Przestacki in \cite{przstacki1,przstacki2,przstacki3}. We remark that the characterization of the symbols is valid for the several variables case.

Recall that $S(\R)$ consists of those smooth functions $f:\R\to\R$ with the property that
$$ \pi_n(f):=\sup_{x\in \R}\sup_{1\leq j\leq n}(1+|x|^2)^n|f^{(j)}(x)|<\infty\ $$ for each $n\in\N.$ $S(\R)$ is a Fr\'echet space when endowed with the topology generated by the sequence of seminorms $\left(\pi_n\right)_{n\in \N}.$ Usually  the Schwartz space refers to the space of complex valued functions satisfying this conditions, which is of fundamental importance in harmonic analysis. If we denote this space by $S_\C(\R)$, then we have $S_\C(\R)=S(\R)\oplus iS(\R)$. Since our objective is to study the behaviour of the composition operator, we can restrict to the real space and the conclusions remain valid for the corresponding composition operator defined in $S_\C(\R)$.

\section{Characterization of symbols.}

\begin{definition}
 A function $\varphi\in C^{\infty}(\R)$ is called a {\em symbol} for $S(\R)$ if  $f\circ \varphi\in S(\R)$ whenever $f\in S(\R)$.
\end{definition}

 If $\varphi$ is a symbol, then the composition operator $C_{\varphi}:S(\R)\to S(\R)$, $f\mapsto f\circ\varphi,$ is continuous by the  closed graph theorem. Our first aim is to characterize the symbols for $S(\R)$ and we begin with a simple necessary condition which, in particular, implies that every symbol $\varphi$ for $S(\R)$ is a semiproper map, that is, for every compact set $K\subset \R$ there exists a compact set $L\subset \R$ such that $\varphi(\R)\cap K \subset \varphi(L).$
\begin{lemma}
\label{lemma:1}
If $\varphi$ is a symbol for $S(\R)$ then
$$\lim_{|x|\to\infty} |\varphi(x)|=\infty.$$ Consequently, either $\varphi(\R) = \R$ or $\varphi(\R) = [a, +\infty)$ for some $a$ or $\varphi(\R) = (-\infty, b]$ for some $b.$
\end{lemma}
\begin{proof}
We proceed by contradiction. Assume that there  exists a sequence $(x_j)$ with $\begin{displaystyle}\lim_{j\to \infty}|x_j|=\infty\end{displaystyle}$ and $\ell\in\R$ such that $\begin{displaystyle}\lim_{j\to\infty} \varphi(x_j) = \ell\end{displaystyle}$. Let $f$ be a compactly supported function such that $f(\ell)=1$. Then $f\in S(\R),$ however $f\circ \varphi\notin S(\R)$ since

$$\lim_{j\to \infty}\left|x_j\left(f\circ\varphi\right)(x_j)\right| = \infty.$$
\end{proof}

For the proof of the next result we recall Fa\`{a} di Bruno formula (see e.g.
\cite[1.3.1]{krantz}):
$$
(f\circ \varphi)^{(n)}(x) = \sum \frac{n!}{k_1!\ldots
k_n!}f^{(k)}(\varphi(x))\left(\frac{\varphi'(x)}{1!}\right)^{k_1}\ldots
\left(\frac{\varphi^{(n)}(x)}{n!}\right)^{k_n}$$
\noindent
where the sum is extended over all $(k_1,\ldots,k_n)\in\N_0^n$ such that
$k_1 + 2k_2 + \ldots + nk_n = n$ and $k:=k_1 + \ldots + k_n.$

\begin{theo}
\label{symbols}
A function $\varphi\in C^{\infty}(\R)$ is a symbol for $S(\R)$ if and only if the following conditions are satisfied:
\begin{itemize}
\item[(i)] For all $j\in \N_0$ there exist $C,p>0$ such that
$$\left|\varphi^{(j)}(x)\right|\leq C(1+\left|\varphi(x)\right|^2)^p$$

\noindent for every $x\in\R$.

\item[(ii)] There exists $k>0$ such that $|\varphi(x)|\geq |x|^{1/k}$ for all $|x|\geq k.$

\end{itemize}

\end{theo}

\begin{proof}
Let us first assume that $\varphi$ is a symbol for $S(\R)$ and prove that condition (i) is satisfied. If this is not the case, there exist $n\in\N$ and a sequence $(x_j)_j\subset \R$ such that

$$|x_j|+1<|x_{j+1}|$$ and

$$\left|\varphi^{(n)}(x_j)\right|\geq j(1+\left|\varphi(x_j)\right|^2)^j$$ for every $j\in {\mathbb N}_0.$

\noindent By Lemma \ref{lemma:1} we can assume, taking a subsequence if necessary,

$$|\varphi(x_j)|+1<|\varphi(x_{j+1})|,\ j\in \N_0.$$ Let us consider $\rho\in \mathcal{D}[-\frac{1}{2},\frac{1}{2}]$ with $\rho'(0)=1$ and $\rho^{(j)}(0)=0$ for $2\leq j\leq n,$ and define

$$f(x):=\sum_{j=1}^{\infty}\frac{\rho(x-y_j)}{(1+|y_j|^2)^j}, \ y_j:=\varphi(x_j).$$ The terms of the sum defining $f$ are disjointly supported. Moreover

$$\supp f\subseteq \bigcup_{j=1}^\infty [y_j-\frac{1}{2},y_j+\frac{1}{2}].$$ If $x\in I_k:=[y_k-\frac{1}{2},y_k+\frac{1}{2}]$ then

$$f^{(j)}(x)=\frac{\rho^{(j)}(x-y_k)}{(1+|y_k|^2)^k}.$$ Hence, for each $j,m\in\N$ there is $C>0$ not depending on $k$ such that

\begin{equation}
\label{eq1}
(1+|x|^2)^m\left|f^{(j)}(x)\right|\leq \|\rho^{(j)}\|_{\infty}\frac{(1+|x|^2)^m}{(1+|y_k|^2)^k}\leq C (1+|y_k|^2)^{m-k},
\end{equation}

\noindent from where it follows

$$\lim_{|x|\to \infty}(1+|x|^2)^m\left|f^{(j)}(x)\right|=0,$$

\noindent i.e. $f\in S(\R)$. Now, an easy consequence of the Fa\`a di Bruno formula gives

$$\left|(f\circ\varphi)^{(n)}(x_k)\right|=\left|\frac{\rho'(0)\varphi^{(n)}(x_k)}{(1+|y_k|^2)^k}\right|=\left|\frac{\varphi^{(n)}(x_k)}{(1+|\varphi(x_k)|^2)^k}\right|\geq k,$$

\noindent and then $f\circ \varphi\not\in S(\R)$, contradicting that $\varphi$ is a symbol for $S(\R)$.

Assume that  $\varphi$ is a symbol for $S(\R)$ but (ii) does not hold. Let $(x_j)_j$ be a sequence in $\R$ such that $|x_j|\geq j$ and $|\varphi(x_j)|^j\leq |x_j|$. We take $\rho$ as above. By Lemma \ref{lemma:1}, we can assume  without loss of generality,
$$
|\varphi(x_j)|+1<|\varphi(x_{j+1})|,\ j\in \N_0.
$$ Now we define

$$f(x):=\sum_{j=1}^\infty\frac{\rho(x-y_j)}{|y_j|^j}, \ y_j:=\varphi(x_j).$$

\noindent We have $f\in S(\R)$ since there exists $C$ independent on $k$ such that $(1+x^2)\leq C(1+ y_k^2)$ for each $x\in I_k = [y_k-\frac{1}{2},y_k+\frac{1}{2}]$. However

$$\liminf_{j\to \infty} |x_j|\cdot\left|\left(f\circ\varphi\right)(x_j)\right| = \liminf_{j\to \infty} |x_j|/|\varphi(x_j)|^j\geq 1,$$

\noindent and hence $f\circ\varphi$ is not in $S(\R)$, a contradiction since $\varphi$ is a symbol for $S(\R)$.

Let us now assume that $\varphi$ satisfies conditions (i) and (ii) and prove that $\varphi$ is a symbol for $S(\R).$ Condition (ii) in the statement implies the boundedness of the function $h(x):=(1+|x|^2)/(1+|\varphi(x)|^2)^k$. Thus, for a fixed $n\in\N$, we can get $C>0$ and $p > 0$ satisfying (i) for $1\leq j\leq n$ and also $|h(x)|^n\leq C$ for all $x\in\R$. This implies that, 
\begin{equation}
\label{1}
(1+x^2)^n\leq  C(1+|\varphi(x)|^2)^{kn}
\end{equation}

\noindent and

 \begin{equation}
 \label{2}
 \left|\varphi^{(j)}(x)\right|\leq C(1+|\varphi(x)|^2)^{p}
 \end{equation}
\noindent for every $1\leq j\leq n$ and $x\in \R.$ 

 \noindent From condition (\ref{2}) and Fa\`a di Bruno formula we get constants $M > 0$ and $t\in \N$,  such that, for each $1\leq j\leq n$ and $f\in S(\R)$

  \begin{equation}
 \label{4}
 \left|(f\circ\varphi)^{(j)}(x)\right|\leq M\sum_{1\leq \ell\leq j}(1+|\varphi(x)|^2)^t\left|f^{(\ell)}(\varphi(x))\right|.
 \end{equation}

\noindent We assume without loss of generality $t>n$. Hence
$$
\sup_{x\in \R}\sup_{1\leq j\leq n}\left(1+x^2\right)^n\left|\left(f\circ \varphi\right)^{(j)}(x)\right| 
$$ is less than or equal to 
$$ MC \sup_{x\in \R}\sum_{1\leq \ell\leq n}(1+|\varphi(x)|^2)^{kn+t}\left|(f^{(\ell)}(\varphi(x))\right|\leq  
nMC \pi_{kn+t}(f) < \infty.
$$ and we conclude $f\circ\varphi\in S(\R).$
\end{proof}

\begin{remark}
Lemma \ref{lemma:1} and Theorem \ref{symbols} remain valid for  symbols in $S(\R^n).$ This is a consequence of the fact that there is $\rho\in\mathcal{D}(\R^n)$ with $\partial^{\alpha}\rho(0)=1$ for $|\alpha|=1,$ $\rho^{(\alpha)}(0)=0$ for $|\alpha|>1,$ as follows from the Borel theorem together with the existence of compactly supported $C^\infty$ functions. Thus, test functions and constant symbols are examples of symbols for $S(\R^d)$ which do not belong to $O_M(\R^d)$.
\end{remark}

From the conditions in Theorem \ref{symbols} the following examples are followed.

\begin{example}
\begin{itemize}
\item[(a)] Non constant polynomials are symbols for $S(\R).$ Also, if $P(x)$ is a polynomial with $\begin{displaystyle}\lim_{|x|\to +\infty}
P(x) = +\infty\end{displaystyle}$ then $$\varphi(x) = \exp\left(P(x)\right)$$ is a symbol for $S(\R)$.
\item[(b)] The symbols for $S(\R)$ are stable under products. Moreover, if $\varphi, \psi$ are symbols and $|\psi(x)| \leq c|\varphi(x)|$ for some $0 < c < 1$ then also $\varphi + \psi$ is a symbol.
\end{itemize}
\end{example}

\section{Compactness of composition operators}

We recall that a continuous operator $T:E\to F$ between two locally convex spaces is said to be compact if there is some $0$-neighbourhood $U$ in $E$ with the property that $T(U)$ is a relatively compact subset of $F.$

We now prove that a composition operator acting on $S(\R)$ is never compact. This contrasts with the behavior of composition operators in Banach spaces of analytic functions. In fact, compactness of composition operators has been extensively studied on various spaces of analytic functions.

The key for the lack of compactness in our setting is the following well-known lemma, the proof of which is included for the convenience of the reader. For a bounded function $f:\R\to\R$  we denote $$\|f\|_\infty=\sup\{|f(x)|:\ x\in\R\}.$$

\begin{lemma}
\label{lemma:2}
For $f\in \mathcal{D}[a,b]$ and $n\in \N$ let define $\|f\|_n:=\sum_{k=0}^{n}\|f^{(k)}\|_{\infty}$.  The norms $\|\cdot\|_n$ and $\|\cdot\|_{n+1}$ are not equivalent. Moreover the system of norms $(\|\cdot\|_n)_{n\in\N}$ defines the topology on $\mathcal{D}[a,b]$.
\end{lemma}
\begin{proof} We argue by contradiction and assume that $\|\cdot\|_n$ and $\|\cdot\|_{n+1}$ are equivalent norms. Then, from $\|f^{(k)}\|_{\infty}=\|(f')^{(k-1)}\|_{\infty}$ and proceeding inductively we conclude that $\|\cdot\|_k$ is equivalent to $\|\cdot\|_{k+1}$ for $k\geq n.$ This is a contradiction since $\mathcal{D}[a,b]$ is not normable.
\end{proof}

\begin{proposition}\label{prop:compact}
Let $\varphi\in C^{\infty}(\R)$ be given and assume $\varphi([a,b])=[c,d]$ and $|\varphi'(x)|\geq \delta>0$ for all $x\in [a,b]$. Then $C_{\varphi}:\mathcal{D}[c,d]\to C^{\infty}(\R),$ $f\mapsto f\circ \varphi,$ is not compact.
\end{proposition}
\begin{proof}
Let $n\in\N$ be fixed. The Fa\'a Di Bruno formula implies that there exist polynomials $Q_1,Q_2,\ldots,Q_{n-1}$ of $n$ variables such that
$$
\begin{array}{ll}
\big(f\circ\varphi\big)^{(n)}(x) & =f^{(n)}\big(\varphi(x)\big)\big(\varphi'(x)\big)^n \\ & \\ & + \sum_{m=1}^{n-1} f^{(m)}\big(\varphi(x)\big)Q_m\big(\varphi'(x),\varphi''(x),\ldots, \varphi^{(n)}(x)\big).\end{array}$$

This implies the existence of $\lambda_n>0$ such that, for all $x\in [a,b],$
\begin{equation}
\label{5}
\begin{array}{ll}
\left|\big(f\circ\varphi\big)^{(n)}(x)\right|& \geq \left|f^{(n)}\big(\varphi(x)\big)\right|\delta^n-\lambda_n\sum_{k=0}^{n-1}\|f^{(k)}\|_{\infty}\\ & \\ & = \left|f^{(n)}\big(\varphi(x)\big)\right|\delta^n-\lambda_n\|f\|_{n-1}.
\end{array}
\end{equation}

For each 0-neighbourhood $U$ in $\mathcal{D}[c,d]$ there exist $\eps>0$ and $p\in\N$  such that

$$\{f\in \mathcal{D}[c,d]:\ \|f\|_{p-1}\leq\eps\}\subseteq U.$$

\noindent According to Lemma \ref{lemma:2} there exists a sequence $(f_j)_j\subset \mathcal{D}[c,d]$ such that $\|f_j\|_{p-1}=\eps$ and $$\delta^p\|f_{j}^{(p)}\|_{\infty} > \lambda_p \eps + j.$$  Now we take $x_j\in [a,b]$ such that
$\|f_{j}^{(p)}\|_\infty=\left|f_{j}^{(p)}\big(\varphi(x_j)\big)\right|$. We get from (\ref{5})

$$\sup\{\left|\big(f_j\circ \varphi\big)^{(p)}(x)\right|:\ x\in [a,b]\}\geq \left|\big(f_j\circ\varphi\big)^{(p)}(x_j)\right|\geq j$$ for every $j\in \N,$ from where it follows that $C_{\varphi}(U)$ is unbounded.
\end{proof}

\begin{theo}
\label{te:compact}
Let $X$ be a locally convex space such that $\mathcal{D}(\R)\hookrightarrow X \hookrightarrow C^{\infty}(\R)$ with continuous inclusions and let $\varphi\in C^{\infty}(\R)$ be a non constant function with the property that $f\circ\varphi\in X$ for each $f\in X.$ Then $C_{\varphi}:X\to X$ is not compact. In particular, $C_{\varphi}: S(\R)\to S(\R)$ is not compact for any  symbol $\varphi$ for $S(\R)$.
\end{theo}
\begin{proof}
Since $\varphi$ is non constant there exists $[a,b]\subseteq \R$ such that $\varphi$ restricted to $[a,b]$ is monotonic, and then there exists $\delta>0$ such that $|\varphi'(x)|\geq\delta$ for each $x\in [a,b]$. The hypothesis imply the continuity of the inclusion $i_1: \mathcal{D}[c,d]\hookrightarrow X$ for $[c,d]=\varphi([a,b])$. Proposition \ref{prop:compact} implies that the composition operator $$\hat{C}_{\varphi}:\mathcal{D}[c,d]\to C^{\infty}(\R), f\mapsto f\circ \varphi,$$ is not compact. The conclusion follows since $\hat{C}_{\varphi}$ decomposes as $i_2\circ C_{\varphi}\circ i_1$, with $i_2$ being the continuous inclusion from $X$ into $C^{\infty}(\R)$.

\end{proof}

\section{Closed range composition operators}

Our aim is to obtain necessary or sufficient conditions for the range of the composition operator to be closed in $S(\R)$. We will relate the closed range property of a composition operator $C_\varphi:S(\R)\to S(\R)$ with the closed range property of $C_\varphi:C^\infty(\R)\to C^\infty(\R)$ (characterized by \cite{przstacki2}) and the closed range property of multiplication operators on $S(\R),$ which has been characterized in \cite{bfj2012}. Concrete examples of composition operators on $S(\R)$ lacking the closed range property are provided.

\begin{lemma}
Let $\varphi$ be a symbol for $S(\R)$. If $C_{\varphi}:S(\R)\to S(\R)$ has closed range then $C_{\varphi}:C^\infty(\R)\to C^{\infty}(\R)$ has also closed range.
\end{lemma}

\begin{proof}
Let $(f_n)_n$ be a sequence in $C^{\infty}(\R)$ such that $(f_n\circ \varphi)_n$ is  convergent to $g$ in $C^{\infty}(\R)$. Our aim is to find $f\in C^\infty(\R)$ such that $g = f\circ\varphi.$ For each $k\in\N$ we consider $\chi_k\in \mathcal{D}[-2k,2k]$ with the property that $\chi_k|_{[-k,k]}=1$. Since $\begin{displaystyle}\lim_{|x|\to \infty}|\varphi(x)|=\infty\end{displaystyle}$ there exist $M_k\in \N$ such that $\varphi^{-1}([-2k,2k])\subseteq [-M_k,M_k].$ Then $(h\circ \varphi)\cdot (\chi_k\circ\varphi)\in \mathcal{D}[-M_k,M_k]$ for each $h\in C^{\infty}(\R)$ and $k\in \N.$ Hence we have
$$\lim_{n\to \infty} (f_n\cdot\chi_k)\circ\varphi=\lim_{n\to \infty} (f_n\circ\varphi)\cdot (\chi_k\circ \varphi)= g\cdot(\chi_k\circ\varphi)$$

\noindent in $\mathcal{D}[-M_k,M_k]$, which is a topological subspace of both $C^{\infty}(\R)$ and $S(\R)$. Hence there exists $h_k\in S(\R)$ such that 

$$g\cdot(\chi_k\circ\varphi)=h_k\circ \varphi.$$

\noindent We observe that, for every natural numbers $k\leq q,$ the condition $\left|\varphi(x)\right|\leq k$ implies $h_k\big(\varphi(x)\big) = h_q\big(\varphi(x)\big),$ since
$$g(x) (\chi_k\circ \varphi)(x)=g(x) (\chi_{q}\circ\varphi)(x) = g(x).$$ Consequently we can define $f\big(\varphi(x)\big) := h_k\big(\varphi(x)\big)$ whenever $|\varphi(x)|\leq k.$ It easily follows that $f$ is $C^\infty$ in the interior of $\varphi(\R)$ and $f\big(\varphi(x)\big) = g(x)$ for all $x\in\R.$ In the case $\varphi(\R) = \R$ we are done. In the case $\varphi(\R) = [a, +\infty),$ $f$ admits right derivatives of every order at $a$, and, by Whitney extension  theorem \cite{whitney} we can extend $f$ to a function in $C^{\infty}(\R)$. A similar argument gives the conclusion in the case that $\varphi(\R) = (-\infty, b].$
\end{proof}

In order to obtain examples of composition operators on $S(\R)$ lacking the closed range property we first consider symbols satisfying the additional condition that there exist $(q_j)_j\subset {\mathbb N}$ and positive constants $(C_j)_j$ such that
$$
\left|\varphi^{(j)}(x)\right|\leq C_j\left(1 + x^2\right)^{q_j}\left(1 + \left|\varphi(x)\right|\right)\ \ \ (\ast)
$$ Examples of symbols with property $(\ast):$
\begin{itemize}
 \item[(a)] Each symbol $\varphi\in O_M.$
 \item[(b)] Every smooth function $\varphi$ such that $\varphi(x) = sign(x)e^{|x|}$ or $\varphi(x) = sign(x)e^{x^2}$ for large values of $|x|.$
\end{itemize}

\begin{proposition}
\label{asterisco}
{\rm Let $\varphi$ be a surjective symbol with property $(\ast)$ such that $C_\varphi:S(\R)\to S(\R)$ has closed range. Then $f\in C^\infty({\mathbb R})$ and $f\circ\varphi\in S(\R)$ imply $f\in S(\R).$
}
\end{proposition}
\begin{proof}
 Let $\chi\in {\mathcal D}[-1,1]$ be a test function such that $0\leq \chi\leq 1$ and $\chi(x) = 1$ for $x\in [-\frac{1}{2}, \frac{1}{2}]$ and consider $\chi_k(x):=\chi(\frac{x}{k})$ and $f_k:= f\cdot \chi_k \in S(\R).$ We {\it claim} that ${\mathcal B}:=\left(f_k\circ \varphi\right)_k$ is a bounded set in $S(\R).$ Once the claim is proved we can proceed as follows. Since $S(\R)$ is a Montel space we can assume, passing to a subsequence if necessary, that there exist
 $$
 g:=S(\R)-\lim_{k\to \infty}f_k\circ \varphi.
 $$ Since $C_\varphi$ has closed range then $g = h\circ\varphi$ for some $h\in S(\R).$ The surjectivity of $\varphi$ and the identity $f\circ\varphi = h\circ\varphi$ imply $f = h\in S(\R).$
\par\medskip
 Finally, we check that ${\mathcal B}$ is a bounded set in $S(\R),$ that is,
 \begin{equation}\label{eq:2}
 \sup_{k\in {\mathbb N}}\sup_{x\in{\mathbb R}}\left(1 + x^2\right)^p\cdot \left|(f_k\circ\varphi)^{(\ell)}(x)\right|\end{equation} is finite for every $p$ and $\ell.$ We observe that
 $$
 (f_k\circ\varphi)(x) = (f\circ\varphi)(x)\cdot \chi_k(\varphi(x)),
 $$ and, by Leibniz's formula, (\ref{eq:2}) is less than or equal to a constant times
\begin{equation}\label{eq:3}
 \sup_{k\in {\mathbb N}}\sup_{0\leq n\leq \ell}\sup_{x\in{\mathbb R}}\left(1 + x^2\right)^p\cdot\left|(f\circ\varphi)^{(\ell-n)}(x)\right|\cdot\left|(\chi_k\circ\varphi)^{(n)}(x)\right|.
\end{equation} According to Fa\`a di Bruno formula,
  $$
 (\chi_k\circ\varphi)^{(n)}(x) = \sum_m C_{m,n}\chi_k^{(m_1 +\ldots + m_n)}\big(\varphi(x)\big)\prod_{j=1}^n\left(\varphi^{(j)}(x)\right)^{m_j}
 $$ where $m = (m_1,\ldots, m_n)$ satisfies $m_1 + 2 m_2 + \ldots + n m_n = n.$ It turns out that there are constants $A_n, B_n > 0$ such that
 $$
\begin{array}{ll}
 \left|(\chi_k\circ\varphi)^{(n)}(x)\right| & \begin{displaystyle}\leq A_n \sum_m \frac{1}{k^{m_1+\ldots +m_n}}\prod_{j=1}^n\left|\varphi^{(j)}(x)\right|^{m_j}\end{displaystyle}\\ & \\ & \begin{displaystyle}\leq B_n \sum_m \frac{1}{k^{m_1+\ldots +m_n}}\prod_{j=1}^n\Big(\left(1+x^2\right)^{m_jq_j}\left(1 + \left|\varphi(x)\right|\right)^{m_j}\Big).\end{displaystyle}\end{array}
 $$ Since $\left|\varphi(x)\right|\leq k$ whenever $(\chi_k\circ\varphi)^{(n)}(x)\neq 0$ we conclude that $\left|(\chi_k\circ\varphi)^{(n)}(x)\right|$ is dominated by some polynomial that depends on $n$ but is independent on $k.$ Since $f\circ\varphi\in S(\R)$ then (\ref{eq:3}) is finite and the claim is proved.
\end{proof}

\begin{corollary}\label{cor:nonclosed} Let $\varphi(x) = sign(x)e^{|x|}$ for large values of $|x|.$ Then the range of $$C_\varphi:S(\R)\to S(\R)$$ is not closed.
\end{corollary}
\begin{proof}
 Take $f\in C^\infty(\R)$ such that $f(x) = 0$ for $x < 0$ while $f(x) = \frac{1}{x}$ for $x\geq 1.$ Then $f\circ \varphi\in S(\R)$ while $f\notin S(\R).$
\end{proof}

A multiplier of the space $S(\R)$ is a smooth function $F$ satisfying $$F\cdot S(\R) \subset S(\R).$$ It is known that $F\in C^\infty(\R)$ is a multiplier of $S(\R)$ if, and only if, for each $k\in\N$ there exist $C > 0$ and $j\in\N$ such that
$$
\left|F^{(k)}(x)\right| \leq C \big(1+x^2\big)^j.
$$ The space of multipliers of $S(\R)$ is denoted by $O_M.$ It is obvious that $F\in O_M$ is equivalent to $F^\prime\in O_M.$ For $F\in O_M$ we denote by $M_F$ the multiplication operator
$$
M_F:S(\R)\to S(\R), f\mapsto F\cdot f.
$$ We now present some sufficient conditions, in terms of multipliers, for the closed range property of composition operators.

\begin{proposition}\label{prop:suficiente1}
Let $\varphi\in O_M$ be a symbol for $S(\R)$ such that $M_{\varphi'}:S(\R)\to S(\R)$ has closed range. Then $C_{\varphi}:S(\R)\to S(\R)$ has also closed range.
\end{proposition}
\begin{proof} Let $(f_n)_n\subset S(\R)$ be a sequence such that $(f_n\circ\varphi)_n$ converges to $g$ in $S(\R)$. According to \cite[Lemma 1.1]{bfj2012}, $\varphi^\prime$ does not have flat points in its zero set and then we can apply \cite{przstacki1} to conclude that $C_{\varphi}:C^{\infty}(\R)\to C^{\infty}(\R)$ has closed range. Consequently there is $f\in C^\infty(\R)$ such that $g = f\circ \varphi.$ We aim to prove that $f\in S(\R).$ Since $\varphi$ is a symbol with polynomial growth we have that for each $T>0$ there exists $N>0$ such that

\begin{equation}\label{onecr}
\sup_{x\in\R}(1+|\varphi(x)|^2)^T|f(\varphi(x))|\leq \sup_{x\in\R}(1+ x^2)^N\left|g(x)\right|<\infty.
\end{equation}
\noindent
In the case $\varphi(\R) = \R,$ (\ref{onecr}) is equivalent to
\begin{equation}\label{twocr}
\sup_{x\in\R}(1+|x|^2)^T|f(x)|<\infty.
\end{equation}

\noindent In the case $\varphi(\R)=[a,\infty)$ we can assume without loss of generality that the restriction of $f$ to $(-\infty,a-1)$ is identically null and we also obtain (\ref{twocr}). Since $(f_n\circ\varphi)_n$ converges to $g$ in $S(\R)$ then $\big(M_{\varphi^\prime}(f_n^\prime\circ \varphi)\big)_n$ converges to $g^\prime = \big(f^\prime\circ \varphi\big)\cdot \varphi^\prime$ in $S(\R).$ Consequently
$$
g^\prime \in \overline{M_{\varphi^\prime}\big(S(\R)\big)} = M_{\varphi^\prime}\big(S(\R)\big).
$$ Since the set of zeros of $\varphi^\prime$ is discrete, we can conclude that
$$
f^\prime\circ\varphi\in S(\R).
$$ Moreover, $M_{\varphi^\prime}:S(\R)\to S(\R)$ is an isomorphism into its range, from where it follows that
$$
\lim_{n\to \infty}f_n^\prime\circ\varphi = f^\prime\circ \varphi\ \ \mbox{in}\ \ S(\R).
$$ Now we can proceed inductively to prove that $f^{(k)}\circ \varphi\in S(\R)$ for every $k\in\N.$ Finally, as in (\ref{onecr}) and (\ref{twocr}) we get
$$\sup_{x\in\R}(1 + x^2)^T|f^{(k)}(x)|<\infty$$ for every $T\in N$ and $k\in\N_0.$ The proof is done.
\end{proof}

In \cite{bfj2012}, the multipliers of $S(\R)$ which have closed range are characterized as the functions $F\in O_M$ such that there exist $N,T,c>0$ such that, if we set $I_{x,T}:=[x-1/(1+x^2)^T,x+1/(1+x^2)^T]$, then we have for each $x\in \R$
\begin{itemize}
 \item[(a)] The cardinality of the set $Z(F)\cap I_{x,T}$, the zeros counted with their multiplicity, is smaller than $N.$
\item[(b)] $(1+x^2)^T|F(x)|>c\prod_{i=1}^{k}|x-x_i|$, $(x_i)_{i=1}^{k}$ being the zeros of $F$ in $I_{x,T}$ counting multiplicities. If there are no zeros in $I_{x,T}$ then in the right side one writes only $c$.
\end{itemize}

If $I\subset \R$ is a closed unbounded interval, the space $S(I)$ can be defined in a natural way. For example, if $I = [a, +\infty),$
$$
S(I):=\{f\in C^{\infty}\left(I\right): \ \sup_{x\in I}\sup_{1\leq j\leq n}(1+x^2)^n\left|f^{(j)}(x)\right|<\infty\ \mbox{ for each }n\in \N\}.$$

\noindent We are considering of course the existence of right derivatives in $a$. Then $S(I)$ is a $(FN)$-space. 

\begin{remark} 
\label{multipliers}
The proof of the main theorem of \cite{bfj2012} (with the obvious modifications) permits to characterize the functions $F:I\to \R$ such that $M_F:S(I)\to S(I)$ is well defined and has closed range as those functions satisfying $\left|F^{(j)}(x)\right| \leq C\big(1+x^2\big)^T,$ ($T = T(j)$) for every $j\in\N_0$ and $x\in I$ (multiplier condition) and moreover, (a) and (b) above are satisfied for each $x\in I$, for $I_{x,T}=[x-1/(1+x^2)^T,x+1/(1+x^2)^T]\cap I$
\end{remark}

We see below that, in the case that $\varphi\in O_M$ is not surjective, we only need to require that $C_\varphi:C^\infty(\R)\to C^\infty(\R)$ have closed range and $\varphi^\prime$ satisfies conditions (a) and (b) in some unbounded interval $I.$

\begin{theo}\label{th:suficiente2}
Let $\varphi$ be a non surjective symbol for $S(\R)$ and assume that $C_{\varphi}:C^{\infty}(\R)\to C^{\infty}(\R)$ has closed range and there exists a closed unbounded interval $I$ such that the multiplication operator $M_{\varphi'}:S(I)\to S(I)$ is well defined and it has closed range. Then $C_{\varphi}:S(\R)\to S(\R)$ has also closed range.
\end{theo}
\begin{proof}
 Let us assume that $\varphi(\R)= [a,\infty[$ and $\varphi'$ is a closed range multiplier for $S((-\infty,b])$. Let $(f_n)_n\subset S(\R)$ be a sequence such that $(f_n\circ \varphi)_n$ converges in $S(\R)$. Since
$C_{\varphi}:C^{\infty}(\R)\to C^{\infty}(\R)$ has closed range, there exists $f\in C^{\infty}(\R)$ such that
\begin{equation}
\label{noninjective1}
\lim_{n\to \infty} (f_n\circ \varphi)_n=f\circ\varphi\in S(\R).
\end{equation}

 Our aim is to show that $f\in S(\R)$. By cutting off after multiplying by a convenient test function, we can assume that the support of  $f$ is contained in $[a-1,\infty).$ Calculating the derivative in (\ref{noninjective1}) we get

\begin{equation}\label{noninjective2}
\lim_{n\to \infty} (f_n'\circ \varphi)\cdot\varphi' = \big(f'\circ\varphi\big)\cdot\varphi'\in S(\R).
\end{equation}

\noindent Thus the convergence of (\ref{noninjective2}) is also in $S((-\infty,b])$. From the hypothesis on $\varphi'$ we get that
$\varphi'$ is a multiplier on $S((-\infty,b])$ and then it has polynomial increase as $x\to-\infty$. Hence also $\varphi$ has polynomial increase as $x\to -\infty$. Since moreover $M_{\varphi'}:S\left((-\infty,b]\right)\to S\left((-\infty,b]\right)$ has closed range we have that $\varphi'$ has no flat points in its zero set in $(-\infty,b]$ and $M_{\varphi'}$ is an isomorphism into its image, hence
$$\lim_{n\to \infty} f_n'\circ \varphi = f'\circ\varphi\in S((-\infty,b]).$$

\noindent Inductively we get

\begin{equation}
\label{noninjective3}
\lim_{n\to \infty} f_{n}^{(j)}\circ \varphi = f^{(j)}\circ\varphi\in S((-\infty,b]).
\end{equation}

\noindent We use now the polynomial increase of $\varphi$ in $(-\infty,b]$ to get that for each $k\in \N$ there exist $N\in\N$ and $C > 0$ such that

$$\lim_{x\to -\infty} (1+\varphi(x)^2)^k\left|f^{(j)}(\varphi(x))\right|\leq C \lim_{x\to -\infty}(1+x^2)^N\left|(f^{(j)}\circ \varphi)(x)\right|=0$$

\noindent for each $1\leq j\leq k$. Finally, as in \ref{prop:suficiente1}, from $\varphi(\R)=[a,\infty[,$  $\begin{displaystyle}\lim_{x\to -\infty}\varphi(x) = +\infty\end{displaystyle}$ and $f$ identically null in $(-\infty,a-1]$ we get
$$
\lim_{|x|\to +\infty}\big(1 + x^2\big)^k|f^{(j)}(x)| = 0
$$ for every $1\leq j\leq k.$ In the case that $\varphi'$ is a closed range multiplier for $S([b,\infty))$ the argument is similar.
\end{proof}

From the proof, it follows the following version for surjective symbols.

\begin{theo}\label{th:suficiente2b}
Let $\varphi$ be a surjective symbol for $S(\R)$ such that $C_{\varphi}:C^{\infty}(\R)\to C^{\infty}(\R)$ has closed range. Assume there exist closed unbounded intervals $I_1$, $I_2$ such that $\R\setminus (I_1\cup I_2)$ is bounded and the multiplication operator $M_{\varphi'}:S(I_j)\to S(I_j)$ is well defined and it has closed range, $j=1,2$. Then $C_{\varphi}:S(\R)\to S(\R)$ has also closed range.
\end{theo}

The examples below illustrate that for non surjective symbols we need only one {\em good branch} to have closed range. However, for surjective symbols we need control on both branches.

\begin{example}
\begin{itemize}
 \item[(a)] Consider the function
$$\varphi_1(x):=\left\{\begin{array}{ll} -e^{-x}+2, & x\leq 0 \\  x, & x\geq 1\end{array}\right.$$ and let $\hat{\varphi_1}$ a $C^\infty$ extension of $\varphi_1$ to $\R$, which always exists by Whitney's theorem \cite{whitney}. Then $\hat{\varphi_1}$ is a surjective symbol and $C_{\hat{\varphi_1}}:S(\R)\to S(\R)$ does not have closed range by Proposition \ref{asterisco} (see the proof of Corollary \ref{cor:nonclosed}). Since $\hat{\varphi_1}$ satisfies that for each $x\in\R$ there exists $y\in \hat{\varphi_1}^{-1}(x)\cap \left(\R\setminus (0,1)\right)$, and then $\hat{\varphi_1}'(y)=\varphi'(y)\neq 0$, we conclude from the main theorem in \cite{przstacki1} that $C_{\hat{\varphi_1}}:C^{\infty}(\R)\to C^{\infty}(\R)$ has closed range.
\item[(b)] Let $x_0>1$ such that $x_0e^{-x_0}=\frac{1}{2e}$ and consider the function
$$\varphi_2(x):=\left\{\begin{array}{ll} xe^{-x} & x\leq x_0 \\  (2x_0-x)+\frac{1}{2e} & x\geq 2x_0\end{array}\right.$$ Let us observe that $\varphi_2\left(\R\setminus (x_0,2x_0)\right)\subseteq (-\infty,e^{-1}]$ and $$\varphi_2(x_0) = \varphi_2(2x_0) = \frac{1}{2e}.$$ If $\hat{\varphi_2}$ is an smooth extension of $\varphi_2$ which is real analytic on $(x_0, 2x_0)$, extension that always exists due to a Whitney extension's theorem \cite{whitney}, then $\hat{\varphi_2}$ does not have any flat critical point. A similar approach to that of the previous example gives that $C_{\hat{\varphi_2}}:C^\infty(\R)\to C^\infty(\R)$ has closed range. Now we apply Theorem \ref{th:suficiente2} to conclude that also $C_{\hat{\varphi_2}}:S(\R)\to S(\R)$ has closed range.
\end{itemize}
\end{example}

Finally we obtain some kind of converse of Proposition \ref{prop:suficiente1} and Theorem \ref{th:suficiente2}.

\begin{theo}
\label{4.9}
Let $\varphi$ be a surjective symbol for $S(\R)$ such that $C_{\varphi}:S(\R)\to S(\R)$ has closed range. We assume that there exists $k$ such that $\varphi'(x)\neq 0$ if $|x|>k$. Then there exists a closed unbounded interval $I$ such that the continuous linear mapping $C_{\varphi}:S(\varphi(I))\to S(I)$ is surjective. Moreover $$M_{\varphi'}:S(I)\to S(I)$$ is well defined and it has closed range of codimension less or equal 1.
\end{theo}

\begin{proof}
Since $f(x) \mapsto f(-x)$ is an isometry we assume without loss of generality that $$\lim_{x\to +\infty}\varphi(x)=+\infty\ \mbox{and}\  \lim_{x\to-\infty}\varphi(x)=-\infty.$$ From the hypothesis we can take $a > 0$ such that $\varphi'(x) > 0$ for all $x\in [a,\infty)$  and $\varphi^{-1}\left(\varphi[a,\infty)\right) = [a,\infty).$ We consider $I=[a,\infty)$ and denote $\varphi(I)=[b,\infty).$  Let $\varepsilon>0$ such that $\varphi' > 0$ in $[a-\varepsilon,a+\varepsilon]$ and let $\delta>0$ and $\eta > 0$ such that $\varphi([a-\varepsilon,a])=[b-\delta,b]$ and $\varphi([a,a+\varepsilon]) = [b,b+\eta].$

We  check first that  $C_{\varphi}:S(\varphi(I))\to S(I)$ has closed range. Notice that the mapping is well defined since $\varphi$ is a symbol for $S(\R)$. Assume there is $(f_n)_n$ in $S(\varphi(I))$ such that
$(f_n\circ\varphi)_n$ converges to $g$ in $S(I)$. We aim to obtain $f\in S(\varphi(I))$ such that $g = f\circ \varphi.$ By Seeley-Mityagin's theorem \cite{mit,seeley} there exists a continuous linear extension operator $E:C^{\infty}(I)\to C^{\infty}(\R)$. Let $\rho$ be a test function such that $0\leq\rho\leq 1$ and with the properties that there is $0<t<\eps/2$ with $\rho=1$ in $(a-t,a+t)$ and its support is contained in $(a-\varepsilon/2,a+\varepsilon/2)$. For $f\in C^\infty(I)$, we define $T(f)(x)=E(f)(x)$ if $x\geq a$ and $T(f)(x)=E(f)(x)\rho(x)$ if $x<a$. Now $T:C^{\infty}(I)\to C^{\infty}(\R)$ is a continuous linear extension operator and

\begin{equation}
\label{extsupp}
\mbox{supp }\left(T(C^{\infty}(I))\right)\subseteq (a-\varepsilon/2,\infty).
\end{equation}

In the above formula we mean that the supports of the functions in $T(C^{\infty}(I)))$ are closed intervals contained in $(a-\varepsilon/2,\infty)$.  Since $\varphi^{-1}$ is smooth in $[b-\delta,b+\eta]$  we can use Whitney's extension theorem to consider an extension $\gamma$ of $(\varphi|_{[a-\eps,a+\eps]})^{-1}$ which is smooth in $(-\infty,b+\eta)$. Since $\gamma(b-\delta) = a-\varepsilon$, there is $s>0$ such that $\gamma(x)<a-\eps/2$ if $x\in (b-\delta-2s,b-\delta+2s)$. If we consider $\alpha\in \mathcal{D}(b-\delta-2s,b-\delta+2s)$ such that $0\leq \alpha\leq 1$ and $\alpha = 1$ in $(b-\delta-s,b-\delta+s)$, we can define 

$$h(x):=\left\{\begin{array}{ll}
(\varphi|_{[a-\eps,a+\eps]})^{-1}(x), & x\in [b-\delta,b+\eta)\\
(\alpha\gamma)(x), & x\in (-\infty,b-\delta)
\end{array}\right. .$$

We have that $h:(-\infty, b+\eta)\to {\mathbb R}$ is a smooth extension of $(\varphi|_{[a-\eps,a+\eps]})^{-1}$
which satisfies
 
\begin{equation}
\label{extleft}
h(-\infty,b-\delta)\subseteq (-\infty,a-\varepsilon/2)
\end{equation}

\noindent and also

\begin{equation}
\label{extleft2}
h(x)=0 \ \mbox{for}\ x\leq b-\delta-2s.
\end{equation}

\noindent Now we define the following  function

$$\hat{f_n}(x) :=\left\{\begin{array}{ll}T(f_n\circ\varphi)\circ h(x),&x\in (-\infty,b)\\  f_n(x), & x\in [b,\infty)\end{array} \right.$$

\noindent Notice that differentiability of $f_n$ implies that $\hat{f_n}$ is smooth on a neighbourhood on $b$. We mean, near $b$ we have  $\hat{f_n}(x)=T(f_n\circ\varphi)\circ h(x)$  if $x<b$ and also $\hat{f_n}(x)=f_n(x)=f_n\circ \varphi\circ(\varphi|_{[a-\eps,a+\eps]})^{-1}(x)=T(f_n\circ \varphi)\circ h(x)$ if $x\in [b, b+\eta)$, i.e $\hat{f_n}$ near $b$ is the composition between two smooth functions. By (\ref{extsupp}) and (\ref{extleft}) we can write

$$\hat{f_n}(x)=\left\{\begin{array}{ll}0,&x\in (-\infty,b-\delta)\\ T(f_n\circ\varphi)\circ (\varphi|_{[a-\eps,a]})^{-1}(x), & x\in [b-\delta,b] \\ f_n(x), & x\in [b,\infty)\end{array} \right.$$

 Since $f_n\in S(\varphi(I))$ we get $\hat{f_n}\in S(\R)$.  Moreover we have
 
 $$\left(\hat{f_n}\circ\varphi\right)(x) = \left\{\begin{array}{ll}0,&x\in \varphi^{-1}(-\infty,b-\delta)\\ T(f_n\circ\varphi)(x), & x\in \varphi^{-1}[b-\delta,b] \\ (f_n\circ\varphi) (x), & x\in \varphi^{-1}(b,\infty) = (a, \infty)\end{array} \right.$$

 Since $\varphi$ is semiproper then $K = \varphi^{-1}\left([b-\delta,b]\right)$ is a compact set. From the fact that $T$ is an extension operator it follows that $T(f_n\circ \varphi)$ converges to $T(g)$ in $C^{\infty}(\R).$ In particular, $T(f_n\circ\varphi)$ and all its derivatives converge uniformly to $T(g)$ on the compact set $K.$ 
Together with the convergence of $(f_n\circ\varphi)_n$ to $g$ in $S(I)$ we get that $\hat{f_n}\circ \varphi$ converges to $T(g)$ in $S(\R)$. Hence there exists $\hat{f}\in S(\R)$ such that $T(g)=\hat{f}\circ \varphi$. If we denote by $f$ the restriction of $\hat{f}$ to $\varphi(I)$, we have that $f\in S(\varphi(I))$ and  $g=f\circ\varphi$.
\par\medskip
Each $f\in S(I)$ which is compactly supported satisfies that $f\circ\varphi^{-1}$ is compactly supported in $\varphi(I)$ and $C^{\infty}$. Since $C_{\varphi}:S(\varphi(I))\to S(I)$ has closed range and the range contains all compactly supported functions in $S(I)$ we conclude that it is surjective. We proceed to show that under these conditions $\varphi'$ is a multiplier for $S(I)$ and $M_{\varphi'}:S(I)\to S(I)$ has closed range.

\noindent {\em Claim}.  $(f\circ\varphi)\varphi'\in S(I)$ for each $f\in S(\varphi(I))$

\noindent{\em Proof of the claim.} Since $\varphi$ is a symbol for $S(\R)$, for each $n\in \N$ there exists $N\in\N$ such that
$|\varphi^{(j)}(x)|\leq (1+|\varphi(x))|^{N}$ for $1\leq j\leq n$ and there exists $k\in\N$ such that $|x|\leq |\varphi(x)|^k$ if
$|x|>k$. Applying Fa\'{a} di Bruno formula and Leibnitz rule we can then obtain  that for each $n,k\in\N$ there exist $C, M>0$ such that

$$\sup_{x\in I}(1+|x|^2)^k \left|((f\circ\varphi)\varphi')^{(n)}(x)\right|\leq
C\sup_{t\in  \varphi(I)}\sum_{j=1}^{n}(1+|t|^2)^M|f^{(j)}(t)|,$$

\noindent the right term being bounded since $f\in S(\varphi(I))$. The claim is proved. Surjectivity of $C_{\varphi}:S(\varphi(I))\to S(I)$ permits to conclude that $$M_{\varphi'}:S(I)\to S(I)$$ is well defined, and then continuous by the closed graph theorem.

  Let
$T:C^{\infty}(I)\to C^{\infty}(\R)$  be the continuos linear extension operator used in the first part of the proof. The construction implies  $T(S(I))\subseteq S(\R)$ and then $T:S(I)\to S(\R)$ is continuous by the closed graph theorem.  Define $v\in S(I)'$ as

 $$v(f)=\int_{-\infty}^{\infty}T(f)(t)dt.$$

 \noindent Let $f_0=(1/\sqrt{\pi}) e^{-x^2} \in S(\R)$ with $\int_{-\infty}^{\infty}f_0(t)dt=1$.
For $f\in S(I)$  define

 $$g(x)=\int_{-\infty}^{x}(T(f)(t)-\lambda f_0(t))dt\quad \lambda=v(f).$$

\noindent We have by L'Hopital rule that $g\in S(\R)$ and

 $$T(f)=g'+\lambda f_0, \quad \lambda=v(f).$$

 \noindent From this we get that for each $f\in S(\varphi(I))$ there exists $g\in S(\R)$ such that

$$T(f\circ\varphi)=g'+\lambda f_0, \quad \lambda=v(f\circ\varphi).$$

\noindent  If we restrict to $I$ we get

\begin{equation}
\label{decomposition}
(f\circ \varphi )= g'+\lambda f_0,\quad \lambda=v(f\circ\varphi).
\end{equation}

\noindent Surjectivity of $C_{\varphi}:S(\varphi(I))\to S(I)$ permits to get $h\in S(\varphi(I))$ such that $g(x) = \left(h\circ\varphi\right)(x)$ for each $x\in I$. This means that if $f\in S(\varphi(I))$ and $f\circ \varphi\in Ker(v)$ then

$$f\circ \varphi = (h'\circ\varphi)\varphi'\in M_{\varphi'}(S(I)).$$

\noindent Since $S(I)=\{f\circ \varphi:\ f\in S(\varphi(I)\}$ we conclude that $M_{\varphi'}(S(I))$ contains the closed hyperplane $Ker(v)$, and then it is closed.

\end{proof}

\begin{theo}
Let $\varphi$ be a surjective symbol for $S(\R)$ such that $C_{\varphi}:S(\R)\to S(\R)$ has closed range. Assume that there exists $k$ such that $\varphi'(x)\neq 0$ if $|x|>k$. Then $\varphi\in O_M$ and there exist $c>0, T> 0$ such that 
$$
\left|\varphi'(x)\right| \geq c\left(1+x^2\right)^{-T}
$$ for $|x|$ large enough.
\end{theo}

\begin{proof}
The  proof of Theorem \ref{4.9} shows that under the hypothesis there are in fact two intervals, $I_1\subset [0,\infty)$ and $I_2\subset [-\infty,0)$ such that $M_{\varphi'}: S(I_j)\to S(I_j)$ is a well defined closed range operator. Hence we get the conclusion from Remark \ref{multipliers} and the equivalence $\varphi\in O_M$ if and only if $\varphi'\in O_M$.
\end{proof}

\noindent
{\bf Acknowledgements}\par\medskip The research was partially supported by the projects MTM2013-43540-P, GVA Prometeo II/2013/013 and ACOMP/2015/186.

Antonio Galbis: Departament d'An\`{a}lisi Matem\`{a}tica, Universitat de Val\`{e}ncia, 46100 Burjassot (Val\`encia), Spain. E-mail: antonio.galbis@uv.es 
\par
Enrique Jord\'a:
Instituto Universitario de MAtem\'atica Pura y Aplicada IUMPA, Universitat Polit\`ecnica de Val\`encia, Plaza Ferr\'andiz y Carbonell, s/n
E-03801 Alcoy (Alicante), Spain. E-mail: ejorda@mat.upv.es

\end{document}